\documentclass[12pt]{amsart}

\usepackage{graphicx}

\newtheorem{theorem}{Theorem}
\newtheorem{lemma}[theorem]{Lemma}
\newtheorem{proposition}[theorem]{Proposition}
\newtheorem{corollary}[theorem]{Corollary}
\theoremstyle{definition}

\newtheorem{notation}{Notation}
\theoremstyle{remark}
\newtheorem{remark}{Remark}

\newcommand{\dita}{Di\c{t}\v{a}}

\begin{document}


\title{A Fast Fourier Transform for Fractal Approximations}
\author{Calvin Hotchkiss}
\author{Eric S. Weber}
\address{Department of Mathematics, Iowa State University, 396 Carver Hall, Ames, IA 50011}
\email{hotchkis@iastate.edu}
\email{esweber@iastate.edu}
\subjclass[2000]{Primary: ?????; Secondary ?????}
\date{\today}
\begin{abstract}
We consider finite approximations of a fractal generated by an iterated function system of affine transformations on $\mathbb{R}^{d}$ as a discrete set of data points.  Considering a signal supported on this finite approximation, we propose a Fast (Fractal) Fourier Transform by choosing appropriately a second iterated function system to generate a set of frequencies for a collection of exponential functions supported on this finite approximation.  Since both the data points of the fractal approximation and the frequencies of the exponential functions are generated by iterated function systems, the matrix representing the Discrete Fourier Transform (DFT) satisfies certain recursion relations, which we describe in terms of \dita's construction for large Hadamard matrices.  These recursion relations allow for the DFT matrix calculation to be reduced in complexity to $O(N \log N)$, as in the case of the classical FFT.
\end{abstract}
\maketitle



\section{Introduction}

The Fast Fourier Transform (FFT) is celebrated as a significant mathematical achievement (see, for example, \cite{AuTo79}). The FFT utilizes symmetries in the matrix representation of the Discrete Fourier Transform (DFT) \cite{CT65}. For $2^N$ (equispaced) data points on [0,1), the matrix representation of the DFT is given by
 \[  \mathcal{F}_{N} = (e^{-2 \pi i \frac{j k}{2^{N}}})_{jk} \]
where $ 0 \leq j,k < 2^{N}$.  The FFT is obtained from the DFT by a permutation of the columns of $\mathcal{F}_{N}$:
\[ \mathcal{F}_{N} = (e^{-2 \pi i \frac{j \sigma(k)}{2^{N}}})_{jk} \]
for $0 \leq j,k < 2^{N}$, where
\[ \sigma(k) = \begin{cases} 2k & \quad 0 \leq k < 2^{N-1}, \\ 2k + 1 & \quad 2^{N-1} \leq k < 2^{N}. \end{cases} \]
The significance of the permutation is that the permuted matrix can be written in the following block form:
\begin{equation} \label{Eq:FFT1}
\mathcal{F}_{N} P = 
\left(
\begin{array}{rr}
\mathcal{F}_{N-1} & D \mathcal{F}_{N-1} \\
\mathcal{F}_{N-1} & -D \mathcal{F}_{N-1} 
\end{array}
\right)
\end{equation}
where $D$ is a diagonal matrix.  This block form reduces the computational complexity of the associated matrix multiplication; recursively, $\mathcal{F}_{N-1}$ can be permuted and written in block form as well.  Repeated application of the column permutation reduces the computational complexity further, and results in overall complexity $O(N \cdot 2^N)$ .

We take the view in the present paper that the DFT arises naturally in the context of iterated function systems, and the FFT arises as reordering of the iterated function system.  Indeed, consider the following set of generators:
\[ \tau_{0}(x) = \frac{x}{2}; \qquad \tau_{1}(x) = \frac{x + 1}{2}. \]
The invariant set of this IFS is the interval $[0,1]$, and the invariant measure is Lebesgue measure restricted to $[0,1]$.  Consider the approximation for the invariant set given by
\[ \mathcal{S}_{N} := \{ \tau_{j_{N-1}} \circ \tau_{j_{N-2}} \circ \cdots \circ \tau_{j_{1}} \circ \tau_{j_{0}} (0) : j_{k} \in \{0, 1\} \}. \]
This is an approximation in the sense that $[0,1] = \overline{ \cup_{N} \mathcal{S}_{N} }$, but the significance for our purposes is that $\mathcal{S}_{N}$ consists of $2^{N}$ equispaced-points:
\[ \mathcal{S}_{N} = \{ \frac{k}{2^{N}} : k \in \mathbb{Z}, \ 0 \leq k < 2^{N} \}. \]

Define a second iterated function system generated by
\[ \rho_{0}(x) = 2x; \qquad \rho_{1}(x) = 2x + 1. \]
Since these are not contractions, the IFS will not have a compact invariant set, but we consider the finite orbits of $0$ under this IFS just as before.  Define
\[ \mathcal{T}_{N} := \{ \rho_{j_N-1} \circ \rho_{j_{N-2}} \circ \cdots \circ \rho_{j_{1}} \circ \rho_{j_{0}} (0) : j_{k} \in \{0, 1\} \}. \]
Note that
\[ \mathcal{T}_{N} = \{ k : k \in \mathbb{Z}, \ 0 \leq k < 2^{N} \}. \]
With the inherited ordering on $\mathcal{S}_{N}$ and $\mathcal{T}_{N}$ from $\mathbb{R}$, say $\mathcal{S}_{N} = \{ s_{0}, s_{1}, \dots, s_{2^{N} -1} \}$ and $\mathcal{T}_{N} = \{ t_{0}, t_{1}, \dots, t_{2^{N} -1 } \}$, we obtain
\[ \mathcal{F}_{N} = ( e^{- 2 \pi i t_{j} s_{k}} )_{jk}. \]

For $0 \leq k < 2^{N}$, we write $k = \sum_{n=0}^{N-1} j_{n} 2^{n}$ with $j_{n} \in \{0,1\}$.  Then
\begin{equation} \label{Eq:reverse}
\tau_{j_{N-1}} \circ \tau_{j_{N-2}} \circ \dots \circ  \tau_{j_{0}} (0) = \dfrac{k}{2^N} = s_{k}.
\end{equation}
However,
\begin{equation} \label{Eq:preserve}
\rho_{j_{0}} \circ \rho_{j_{1}} \circ \dots \circ \rho_{j_{N-1}} (0) = k = t_{k}.
\end{equation}

We define a new ordering on $\mathcal{S}_{N}$ as follows:
\begin{equation} \label{Eq:preserve2}
\tilde{s}_{k} = \tau_{j_{0}} \circ \tau_{j_{1}} \circ \dots \circ \tau_{j_{N-1}}(0)
\end{equation}
where $k$ is written in base 2. As we shall see in Theorem \ref{Th:main1}, this new ordering on $\mathcal{S}_{N}$ results in the following matrix equality:
\begin{equation} \label{Eq:FFT2}
(e^{-2 \pi i t_{j} \tilde{s}_{k}} )_{jk} = \mathcal{F}_{N}P
\end{equation}
as in Equation (\ref{Eq:FFT1}).

We will call the compositions in Equations (\ref{Eq:preserve}) and (\ref{Eq:preserve2}) the \emph{obverse} ordering.  The composition in Equation (\ref{Eq:reverse}) will be called the \emph{reverse} ordering.  As suggested previously, and will be established in Theorem \ref{Th:main1}, if the elements of $\mathcal{S}_{N}$ and $\mathcal{T}_{N}$ are both ordered with the obverse compositions, then the permuted DFT matrix obtained is as in Equation (\ref{Eq:FFT2}).  However, if both $\mathcal{S}_{N}$ and $\mathcal{T}_{N}$ are ordered using the reverse compositions, then the matrix becomes
\[ ( e^{-2 \pi i \tilde{t}_{j} s_{k} } )_{jk} = P \mathcal{F}_{N} = 
\left(
\begin{array}{rr}
\mathcal{F}_{N-1} & \mathcal{F}_{N-1} \\
\mathcal{F}_{N-1} D & - \mathcal{F}_{N-1} D 
\end{array}
\right),
\]
a block form that will allow the inverse $\mathcal{F}_{N}^{-1}$ to have a fast multiplication algorithm.

Consider the measure $\mu_{N} = \frac{1}{2^{N}} \sum_{s \in \mathcal{S}_{N}} \delta_{s}$; this sequence of measures converges weakly to Lebesgue measure restricted to $[0,1]$, the invariant measure for the IFS generated by $\{ \tau_{0} , \tau_{1} \}$.  Moreover, we consider the exponential functions $\{ e^{2 \pi i t (\cdot)} : t \in \mathcal{T}_{N}\} \subset L^2(\mu_{N})$; this set will be an orthonormal basis, and the DFT is the matrix representation of this basis (up to a scaling factor).  Thus, the IFS generated by $\{ \tau_{0}, \tau_{1} \}$ gives rise to a fractal, and the IFS generated by $\{ \rho_{0}, \rho_{1} \}$ gives rise to the frequencies of an orthonormal set of exponentials.

A probability measure $\mu$ is \emph{spectral} if there exists a set of frequencies $\Lambda \subset \mathbb{R}$ such that $\{ e^{2 \pi i \lambda (\cdot)} : \lambda \in \Lambda \} \subset L^2(\mu)$ is an orthonormal basis.  If the measure is spectral, the set $\Lambda$ is called a spectrum for $\mu$.  Jorgensen and Pederson \cite{JP98} prove that the uniform measure supported on the middle-thirds Cantor set is not spectral.  However, they prove that the invariant measure $\mu_{4}$ for the iterated function system generated by
\[ \tau_{0}(x) = \frac{x}{4}, \qquad \tau_{1}(x) = \frac{x+2}{4} \]
is spectral, and moreover, the spectrum is obtained via the iterated function system generated by
\[ \rho_{0}(x) = 4x, \qquad \rho_{1}(x) = 4x + 1. \]
In fact, the orbit of $0$ under the iterated function system generated by $\{ \rho_{0}, \rho_{1} \}$ is a spectrum for $\mu_{4}$.

For a generic iterated function system $\{ \psi_{0}, \dots, \psi_{K-1} \}$ consisting of contractions on $\mathbb{R}^{d}$, we will consider an approximation $\mathcal{S}_{N}$ to the invariant set given by
\[ \mathcal{S}_{N} := \{ \psi_{j_{N-1}} \circ \psi_{j_{N-2}} \circ \cdots \circ \psi_{j_{1}} \circ \psi_{j_{0}} (0) : j_{k} \in \{0, 1, \dots, K-1\} \}. \]
This collection of points we will consider as the locations of data points.  We then will choose a second iterated function system $\{ \rho_{0}, \dots, \rho_{K-1} \}$, and consider the finite orbit of $0$:
\[ \mathcal{T}_{N} := \{ \rho_{j_{N-1}} \circ \rho_{j_{N-2}} \circ \cdots \circ \rho_{j_{1}} \circ \rho_{j_{0}} (0) : j_{k} \in \{0, 1, \dots, K-1\} \}. \]
These will be the frequencies for an exponential basis in $L^2(\mu_{N})$, where $\mu_{N} = \frac{1}{K^{N}} \sum_{s \in \mathcal{S}_{N}} \delta_{s}$.

A necessary and sufficient condition to obtain an exponential basis for $L^2(\mu_{N})$ from the frequencies in $\mathcal{T}_{N}$ is that the matrix 
\[ H_{N} = (e^{-2 \pi i s_{j} t_{k}})_{j,k} \]
is invertible, where $s_{j}$ and $t_{k}$ range through $\mathcal{S}_{N}$ and $\mathcal{T}_{N}$ under any ordering, respectively.  Preferably, the matrix $H_{N}$ would be \emph{Hadamard}, i.e.\ $H_{N}^{*} H_{N} = K^{N} I_{K^{N}}$ (since it automatically has entries of modulus 1), since this would correspond to an orthogonal exponential basis.  As we will show, if $H_{1}$ is invertible (Hadamard) then all $H_{N}$ will be invertible (Hadamard, respectively).

Moreover, we will put an ordering (namely, the obverse ordering) on $\mathcal{S}_{N}$ and $\mathcal{T}_{N}$ so that under this ordering the matrix $H_{N}$ has a block form in the manner of \dita's construction for large Hadamard matrices.  This block form will allow for the computational complexity of the matrix multiplication to be reduced.  Then, $\mathcal{S}_{N}$ and $\mathcal{T}_{N}$ will be reordered (using the reverse ordering) so that the inverse of $H_{N}$ will have a similar block form, again allowing for a fast algorithm for the matrix multiplication.

We note that for a generic IFS, the set $\mathcal{S}_{N}$ will consist of irregularly spaced points.  We view the matrix $H_{N}$ as being a Fourier transform for a signal (or set of data points) located at the points in $\mathcal{S}_{N}$, and thus $H_{N}$ (and its block form as shown in Theorem \ref{Th:main1}) can be considered a non-equispaced FFT.  We further note, however, that this is not a full irregularly spaced FFT, since all of the data point locations in $\mathcal{S}_{N}$ are rationally related.  Please see \cite{DR93a,DR95a,GL04a} for the irregularly spaced FFT.

\subsection{\dita's Construction of Large Hadamard Matrices}

\dita's construction for large Hadamard matrices is as follows \cite{Dita04,TaZy06}.  If $A$ is a $K \times K$ Hadamard matrix, $B$ is an $M \times M$ Hadamard matrix, and $E_1,\dots, E_{K-1}$ are $M \times M$ unitary diagonal matrices, then the $KM \times KM$ block matrix $H$ is a Hadamard matrix:

\begin{equation} \label{Eq:dita}
H =  \left( \begin{array}{cccc} a_{00} B & a_{01}E_1  B & \dots & a_{0(K-1)} E_{K-1} B \\
 a_{10} B & a_{11}E_1  B & \dots & a_{1(K-1)} E_{K-1} B \\
\vdots & \vdots & \ddots & \vdots \\
a_{(K-1)0} B & a_{(K-1)1}E_1  B & \dots & a_{(K-1)(K-1)} E_{K-1} B \\
\end{array}
\right).
\end{equation}

Since we will also consider invertible matrices, not just Hadamard matrices, we show that for $A$, $B$, $E_1,\dots, E_{K-1}$ invertible, $H$ will also be invertible, and its inverse has a similar block form.

\begin{proposition}\label{Hoinv} Suppose $A$ and $B$ are invertible, $E_1,\dots, E_{K-1}$ are invertible and diagonal.  Let $C = A^{-1}$.  For the matrix $H$ in Equation \ref{Eq:dita}, 
\begin{equation} \label{Eq:dita-inverse}
H^{-1} =  
\left( 
\begin{array}{cccc} 
c_{00} B^{-1}  & c_{01}B^{-1} & \dots & c_{0(K-1)} B^{-1} \\
c_{10}  B^{-1} E_1^{-1}  & c_{11}B^{-1}E_1^{-1} & \dots & c_{1(K-1)} B^{-1} E_1^{-1} \\
\vdots & \vdots & \ddots & \vdots \\
c_{(K-1)0} B^{-1}E_{K-1}^{-1} & c_{1(K-1)} B^{-1}E_{K-1}^{-1} & \dots & c_{(K-1)(K-1)} B^{-1} E_{K-1}^{-1} \\
\end{array}
\right).
\end{equation}
\end{proposition}

\begin{proof}
Let $G$ be the block matrix in Equation (\ref{Eq:dita-inverse}), and let $E_{0} = I_{M}$.  Note that the product of $H$ and $G$ will have a block form.  Multiplying the $j$-th row of $H$ with the $\ell$-th column of $G$, we obtain that the $j,\ell$ block of $HG$ is:
\[
\sum_{k=0}^{K-1} (a_{jk} E_{k} B) ( c_{k \ell} B^{-1} E_{k}^{-1} ) = \sum_{k=0}^{K-1} a_{jk} c_{k \ell} I_{M}.
\]
Since $\sum_{k=0}^{K-1} a_{jk} c_{k \ell} = \delta_{j, \ell}$, we obtain $HG = I_{KM}$.
\end{proof}

If $A$, $B_0,\dots,B_{K-1}$, $E_1,\dots, E_{K-1}$ are all unitary, then the construction for $H^{-1}$ gives $H^*$, so $H$ is also unitary.

\subsection{Complexity of Matrix Multiplication in \dita's Construction}

Let $\vec v$ be a vector of length $KM$. Consider $H \vec v$ where $H$ is the block matrix as in Equation (\ref{Eq:dita}).  We divide the vector $\vec{v}$ into $K$ vectors of length $M$ as follows:
\[ \vec{v} = \begin{pmatrix} \vec{v}_{0} \\ \vec{v}_{1} \\ \vdots \\ \vec{v}_{K-1} \end{pmatrix} .
\]
Then the matrix multiplication $H \vec{v}$ can be reduced in complexity, since 
\[ H \vec{v} = 
\begin{pmatrix}
\sum_{j=0}^{K-1} a_{0j} E_{j} B \vec{v}_{0} \\
\sum_{j=0}^{K-1} a_{1j} E_{j} B \vec{v}_{1} \\
\vdots \\
\sum_{j=0}^{K-1} a_{(K-1) j} E_{j} B \vec{v}_{K-1}
\end{pmatrix}.
\]

Let $\mathcal{O}_{M}$ be the number of operations required to multiply the a vector $\vec{w}$ of length $M$ by the matrix $B$.  The total number of operations required for each component of $H \vec{v}$ is $\mathcal{O}_{M} + M(K-1) + MK$ multiplications and $M(K-1)$ additions.  The total number of operations for $H\vec{v}$ is then $K \mathcal{O}_{M} + 3MK^2 - 2MK$.  We have just established the following proposition.

\begin{proposition}\label{Hocount} The product $H \vec v$ requires at most $K \mathcal{O}_{M} + 3MK^2 - 2MK$ operations.
\end{proposition}

Since $\mathcal{O}_{M} = O(M^2)$, we obtain that the computational complexity of $H$ is $O(M^2K + MK^2)$, whereas for a generic $KM \times KM$ matrix, the computational complexity is $O(K^2M^2)$.  Thus, the block form of $H$ reduces the computational complexity of the matrix multiplication.

\section{A Fast Fourier Transform on $\mathcal{S}_{N}$}

We consider an iterated function system generated by contractions $\{ \psi_{0}, \psi_{1}, \dots, \psi_{K-1} \}$ on $\mathbb{R}^{d}$ of the following form:
\[ \psi_{j}(x) = A(x + \vec{b}_{j}) \]
where $A$ is a $d \times d$ invertible matrix with $\| A \| < 1$.  We require $A^{-1}$ to have integer entries, the vectors $\vec{b}_{j} \in \mathbb{Z}^d$, and without loss of generality we suppose $\vec{b}_{0} = \vec{0}$.  We then choose a second iterated function system generated by $\{ \rho_{0}, \rho_{1}, \dots, \rho_{K-1} \}$ of the form
\[ \rho_{j}(x) = Bx + \vec{c}_{j} \]
where $B = (A^{T})^{-1}$, with $\vec{c}_{j} \in \mathbb{Z}^{d}$, and $\vec{c}_{0} = \vec{0}$.  We require the matrix
\[ M_{1} = ( e^{- 2 \pi i  \vec{c}_{j} \cdot A\vec{b}_{k}} )_{j,k} \]
be invertible (or Hadamard).  Note that depending on $A$ and $\{ \vec{b}_{0}, \vec{b}_{1}, \dots, \vec{b}_{K-1} \}$, there may not be any choice $\{ \vec{c}_{0}, \vec{c}_{1}, \dots , \vec{c}_{K-1} \}$ so that $M_{1}$ is invertible.  However, for many IFS's there is a choice:
\begin{proposition} \label{P:invertible}
If the set $\{\vec{b}_{0}, \vec{b}_{1}, \dots, \vec{b}_{K-1} \}$ is such that for every pair ($j \neq k$), $A\vec{b}_{j} - A \vec{b}_{k} \notin \mathbb{Z}^{d}$, then there exists $\{ \vec{c}_{0}, \vec{c}_{1}, \dots, \vec{c}_{K-1} \}$ such that the matrix $M_{1}$ is invertible.
\end{proposition}

\begin{proof}
The mappings $\phi_{1} : \vec{x} \mapsto e^{2 \pi i \vec{x} \cdot A \vec{b}_{j}}$ and $\phi_{2} : \vec{x} \mapsto e^{2 \pi i \vec{x} \cdot A \vec{b}_{k}}$ are characters on $G = \mathbb{Z}^{d} / B \mathbb{Z}^d$.  Since $A \vec{b}_{j} - A \vec{b}_{k} \notin \mathbb{Z}^{d}$, the characters are distinct.  Thus, by Schur orthogonality, $\sum_{x \in G} \phi_{1}(x) \overline{\phi_{2}(x)} = 0$.  Therefore, the matrix $M = ( e^{-2 \pi i \vec{x}_{k} \cdot A \vec{b}_{j}} )_{j,k}$, where $\{ \vec{x}_{k} \}$ is any enumeration of $G$, has orthogonal columns.  Thus, there is a choice of a square submatrix of $M$ which is invertible.
\end{proof}

Even under the hypotheses of Proposition \ref{P:invertible} there is not always a choice of $\vec{c}$'s so that $M_{1}$ is Hadamard; this is the case for the middle-third Cantor set, which is the attractor set for the IFS generated by $\psi_{0}(x) = \frac{x}{3}$, $\psi_{1}(x) = \frac{x + 2}{3}$ (and is a reflection of the fact that $\mu_{3}$ is not spectral).

\begin{notation} We define our notation for compositions of the IFS's using two distinct orderings. Let $N \in \mathbb{N}$. 
For $j \in \{0,1,\dots,K^{N}-1\}$, write $j = j_0 + j_1 K + \dots + j_{N-1} K^{N-1}$ with $j_0,\dots, j_{N-1} \in \{0,1,\dots K-1\}$.   We define
\begin{align*}
\Psi_{j,N} &:= \psi_{j_{0}} \circ \psi_{j_{1}} \circ \dots \circ \psi_{j_{N-1}} \\
\mathcal{R}_{j,N} &:= \rho_{j_{0}} \circ \rho_{j_{1}} \circ \dots \circ \rho_{j_{N-1}}.
\end{align*}
These give rise to enumerations of $\mathcal{S}_{N}$ and $\mathcal{T}_{N}$ as follows:
\begin{align*}
\mathcal{S}_{N} &= \{ \Psi_{j,N}(0) : j = 0, 1, \dots K^{N} -1 \} \\
\mathcal{T}_{N} &= \{ \mathcal{R}_{j,N}(0) : j = 0, 1, \dots K^{N} -1 \}.
\end{align*}
We call these the ``obverse'' orderings of $\mathcal{S}_{N}$ and $\mathcal{T}_{N}$.

Likewise, we define
\begin{align*}
\widetilde \Psi_{j,N} &:= \psi_{j_{N-1}} \circ \psi_{j_{N-2}} \circ \dots \circ \psi_{j_{0}} \\
\mathcal{\widetilde{R}}_{j,N} &:= \rho_{j_{N-1}} \circ \rho_{j_{N-2}} \circ \dots \circ \rho_{j_{0}}
\end{align*}
which also enumerate $\mathcal{S}_{N}$ and $\mathcal{T}_{N}$.  We call these the ``reverse'' orderings. 
\end{notation}

\begin{remark} \label{R:scale1}
Note that for $N=1$, $\Psi_{j,1} = \widetilde{\Psi}_{j,1}$ and $\mathcal{R}_{j,1} = \widetilde{\mathcal{R}}_{j,1}$.
\end{remark}

We define the matrices  $M_{N}$ and $\widetilde{M}_{N}$ as follows:
\[ [M_{N}]_{jk} = e^{-2 \pi i \mathcal{R}_{j,N}(0) \cdot \Psi_{k,N}(0) } \]
and
\[ [\widetilde{M}_{N}]_{jk} = e^{-2 \pi i \widetilde{\mathcal{R}}_{j,N}(0) \cdot \widetilde{\Psi}_{k,N}(0) }. \]
Both of these are the matrix representations of the exponential functions with frequencies given by $\mathcal{T}_{N}$ on the data points given by $\mathcal{S}_{N}$.  The matrix $M_{N}$ corresponds to the obverse ordering on both $\mathcal{T}_{N}$ and $\mathcal{S}_{N}$, whereas the matrix $\widetilde{M}_{N}$ corresponds to the reverse ordering on both.  Since these matrices arise from different orderings of the same sets, there exist permutation matrices $P$ and $Q$ such that
\begin{equation} \label{Eq:permute}
Q \widetilde{M}_{N} P = M_{N}.
\end{equation}

Indeed, define for $j \in \{0, \dots, K^N - 1\}$ a conjugate as follows: if $j = j_{0} + j_{1} K + \dots + j_{N-1} K^{N-1}$, let $\tilde{j} = j_{N-1} + j_{N-2} K + \dots + j_{0} K^{N-1}$.  Note then that $\tilde{\tilde{j}} = j$, and
\begin{equation} \label{Eq:tilde}
\widetilde{\Psi}_{k, N} = \Psi_{\tilde{k} , N} \qquad \widetilde{\mathcal{R}}_{k, N} = \mathcal{R}_{\tilde{k}, N}. 
\end{equation}
Now, define a $K^N \times K^N$ permutation matrix $P$ by $[P]_{mn} = 1$ if $n=\tilde{m}$, and $0$ otherwise.
\begin{lemma}  \label{L:tilde}
For $P$ defined above,
\[ P \widetilde{M}_{N} P = M_{N}. \]
\end{lemma}
\begin{proof}
We calculate
\begin{align*}
[P \widetilde{M}_{N} P]_{mn} &= \sum_{k} [P]_{mk} \sum_{\ell} [\widetilde{M}_{N}]_{k \ell} [P]_{\ell n} \\
&= [P]_{m \tilde m} [\widetilde{M}_{N}]_{\tilde{m} \tilde{n}} [P]_{\tilde{n} n} \\
&= e^{- 2 \pi i \widetilde{\mathcal{R}}_{\tilde{m}, N}(0) \cdot \widetilde{\Psi}_{\tilde{n}, N} (0) }\\
&= e^{- 2 \pi i \mathcal{R}_{m, N}(0) \cdot \Psi_{n, N} (0) } = [M_{N}]_{mn}
\end{align*}
by virtue of Equation (\ref{Eq:tilde}).
\end{proof}

\begin{proposition} \label{P:m1}
For scale $N=1$,
\[
M_1 = \widetilde{M}_1= \left( \begin{array}{cccc} 1 & 1 & \dots & 1 \\
1 & \exp (2 \pi i \vec  c_1 \cdot A \vec b_1) & \dots & \exp (2 \pi i\vec c_1 \cdot A \vec b_{K-1}) \\
\vdots & \vdots & \vdots \ \vdots & \vdots \\
1 & \exp (2 \pi i \vec c_{K-1} \cdot A \vec b_1) & \dots & \exp (2 \pi i \vec c_{K-1} \cdot A \vec b_{K-1}) \\
\end{array}
\right). 
\]
\end{proposition}

\begin{proof}
The proof follows from Remark \ref{R:scale1}.
\end{proof}

\begin{lemma}\label{lem1}
For $N \in \mathbb{N}$, $0 \leq j < K^{N}$, and $\vec{x}, \vec{y} \in \mathbb{R}^d$,
\begin{enumerate}
\item $\Psi_{j,N}\left(\vec x + \vec y \right) = \Psi_{j,N} (\vec x) + A^N \vec y$ \label{lem1a}
\item $\widetilde{\Psi}_{j,N}\left(\vec x + \vec y \right) = \widetilde\Psi_{j,N} (\vec x) + A^N \vec y$ \label{lem1b}
\item $\mathcal{R}_{j,N}\left(\vec{x} + \vec{y}\right) = R_{j,N}(\vec{x}) + B^{N}\vec{y}$ \label{lem1c}
\item $\widetilde{\mathcal{R}}_{j,N}\left(\vec{x} + \vec{y}\right) = \widetilde{R}_{j,N}(\vec{x}) + B^{N}\vec{y}$. \label{lem1d}
\end{enumerate}
\end{lemma}

\begin{proof} 
We prove by induction on $N$. The base case is easily checked.  Assume the equality in Item i) holds for $N-1$.  For $j = j_{0} + j_{1} K + \dots + j_{N-1} K^{N-1}$, let $\ell = j - j_{N-1} K^{N-1}$.  We have:
\begin{align*}
\Psi_{j,N}\left(\vec x + \vec y \right) &= \Psi_{\ell, N-1} ( \psi_{j_{N-1}}(\vec{x} + \vec{y}) ) \\
&= \Psi_{\ell, N-1} ( \psi_{j_{N-1}}(\vec{x}) + A \vec{y} ) \\
&= \Psi_{\ell, N-1} ( \psi_{j_{N-1}}(\vec{x})) + A^{N-1} A \vec{y}  \\
&= \Psi_{j,N} (\vec x )+ A^N\vec y
\end{align*}

The proofs for the other three identities are analogous. 
\end{proof}

\begin{lemma}\label{lem2} 
For $N \in \mathbb{N}$ and $0 \leq j < K^{N}$,
\begin{enumerate}
\item $\Psi_{j,N}(0) = A^N \vec z$ for some $\vec z \in \mathbb{Z}^d$,
\item $\widetilde\Psi_{j,N}(0) = A^N \vec z$ for some $\vec z \in \mathbb{Z}^d$,
\item $\mathcal{R}_{j,N}(0) \in \mathbb{Z}^d$,
\item $\widetilde{\mathcal{R}}_{j,N}(0) \in \mathbb{Z}^d$.
\end{enumerate}
\end{lemma}

\begin{proof}
We prove by induction on $N$. the base case is easily checked.  Assume the equality in Item i) holds for $N-1$.  For $j = j_{0} + j_{1} K + \dots + j_{N-1} K^{N-1}$, let $q_{j} = j - j_{N-1} K^{N-1}$.  We have:
\begin{align*}
\Psi_{j,N} ( 0 ) &= \psi_{j_{N-1}}\left( \Psi_{q_{j},N-1} ( 0 ) \right) \\
&= A \left( A^{N-1} \vec z + \vec b_j \right) \\
&= A^{N}(\vec z + A^{-(N-1)} \vec b_j)
\end{align*}
Since $A^{-1}$ is an integer matrix, so is $A^{-(N-1)}$ and thus $\vec z + A^{-(N-1)} \vec b_j \in \mathbb{Z}^d$.  Item ii) is analogous.   For Item iii), note first that $\rho_{j}(\mathbb{Z}^d) \subset \mathbb{Z}^d$, so by induction, $ \rho_{j_{0}} \circ \dots \circ \rho_{j_{N-1}}(0) \in \mathbb{Z}^d$.  Likewise for Item iv).
\end{proof} 

\begin{lemma}  \label{lem3} Assume $N \geq 2$, let $\ell$ be an integer between $0$ and $K-1$, and suppose $l \cdot K^{N-1} \leq j < (l+1) K^{N-1}$.  Then,
\begin{enumerate}
\item $\Psi_{j,N} (0)  = \Psi_{j-l \cdot K^{N-1}, N-1} (0) + A^{N} \vec b_l$,
\item $\widetilde{\Psi}_{j,N} (0)  = A \widetilde{\Psi}_{j-l \cdot K^{N-1}, N-1} (0) + A \vec b_l$,
\item $\mathcal{R}_{j,N} (0)  = \mathcal{R}_{j-l \cdot K^{N-1}, N-1} (0) + B^{N-1} \vec c_l$,
\item $\widetilde{\mathcal{R}}_{j,N} (0)  = B \widetilde{\mathcal{R}}_{j-l \cdot K^{N-1}, N-1} (0) + \vec c_l$.
\end{enumerate}
\end{lemma}

\begin{proof}
For $l \cdot K^{N-1} \leq j < (l + 1) K^{N-1}$, $j_{N-1} = l$, so we have:
\begin{align*}
\Psi_{j,N} (0) &=  \psi_{j_0} \circ \psi_{j_1} \circ \dots \circ \psi_{j_{N-2}} \circ \psi_l (0) \\
&=  \psi_{j_0} \circ \psi_{j_1} \circ \dots \circ \psi_{j_{N-2}} \left( A ( 0 + \vec b_l ) \right) \\
&= \Psi_{j - l \cdot K^{N-1},N -1} (0 + A \vec b_l).
\end{align*}
Applying Lemma \ref{lem1} Item i) to $\Psi_{j - l \cdot K^{N-1},N -1}$:
\[
 \Psi_{j - l \cdot K^{N-1},N -1} (0 + A \vec b_l) = \Psi_{j - l \cdot K^{N-1},N -1} (0) + A^{N-1} A \vec b_l.  
\]
The proof of Item iii) is similar to Item i) with one crucial distinction, so we include the proof here.  We have:
\begin{align*}
\mathcal{R}_{j,N} (0) &=  \rho_{j_0} \circ \rho_{j_1} \circ \dots \circ \rho_{j_{N-2}} \circ \rho_l (0) \\
&=  \rho_{j_0} \circ \rho_{j_1} \circ \dots \circ \rho_{j_{N-2}} \left( B 0 + \vec c_l\right) \\
&= \mathcal{R}_{j - l \cdot K^{N-1},N -1} (0 + \vec c_l).
\end{align*}
Applying Lemma \ref{lem1} Item iii) to $R_{j - l \cdot K^{N-1},N -1}$:
\[
 \mathcal{R}_{j - l \cdot K^{N-1},N -1} (0 + \vec c_l) = \mathcal{R}_{j - l \cdot K^{N-1},N -1} (0) + B^{N-1} \vec c_l.
\]

For Item ii), we have
\begin{align*}
\widetilde{\Psi}_{j,N} (0) &= \psi_{\ell} ( \widetilde{\Psi}_{j - \ell \cdot K^{N-1} ,N-1} (0) ) \\
&= A \widetilde{\Psi}_{j - \ell \cdot K^{N-1} ,N-1} (0) + A \vec{b}_{\ell}.
\end{align*}
The proof of Item iv) is analogous.
\end{proof}

Note that in Item i), the extra term involves $A^{N}$, whereas in Item iii) the extra term involves $B^{N-1}$.  We are now in a position to prove our main theorem.

\begin{theorem} \label{Th:main1}
The matrix $M_{N}$ representing the exponentials with frequencies given by $\mathcal{T}_{N}$ on the fractal approximation $\mathcal{S}_{N}$, when both are endowed with the obverse ordering, has the form
\begin{equation} \label{Eq:main1}
M_{N} = 
\left( \begin{array}{cccc} m_{00} M_{N-1} & m_{01} D_{N,1}  M_{N-1} & \dots & m_{0(K-1)} D_{N,K-1} M_{N-1} \\
 m_{10} M_{N-1} & m_{11} D_{N,1}  M_{N-1} & \dots & m_{1(K-1)} D_{N,K-1} M_{N-1} \\
\vdots & \vdots & \vdots \ \vdots & \vdots \\
m_{(K-1)0} M_{N-1} & m_{(K-1)1} D_{N,1}  M_{N-1} & \dots & m_{(K-1)(K-1)} D_{N,K-1} M_{N-1} \\
\end{array}
\right).
\end{equation}
Here, $D_{N,m}$ are diagonal matrices with $[D_{N,m}]_{pp}= e^{- 2\pi i \mathcal{R}_{p,N-1}(0) \cdot A^N \vec b_m}$, and $m_{jk} = [M_1]_{jk}$. 
\end{theorem}

\begin{proof}
Let us first subdivide $M_{N}$ into blocks $B_{\ell m}$ of size $K^{N-1} \times K^{N-1}$, so that 
\[ M_{N} = \begin{pmatrix} B_{00} & \dots  & B_{0 (K-1)} \\ \vdots & \ddots & \vdots \\ B_{(K-1) 0} & \dots & B_{(K-1) (K-1)}
\end{pmatrix}. \]
Fix $0 \leq j,k < K^{N}$ and suppose $ \ell K^{N-1} \leq j < (\ell + 1)K^{N-1}$ and $m K^{N-1} \leq k < (m+1)K^{N-1}$ with $0 \leq \ell, m < K$.  Let $q_{j} = j - \ell K^{N-1}$ and $q_{k} = k - m K^{N-1}$.  Observe that 
\begin{equation} \label{Eq:matrix1}
[ M_{N} ]_{jk} = [ B_{\ell m} ]_{q_{j} q_{k}}. 
\end{equation}

Using Lemma \ref{lem3} Items ii) and iv), we calculate 
\[ \mathcal{R}_{j,N}(0) \cdot \Psi_{k,N}(0) = \left( \mathcal{R}_{q_{j}, N-1}(0) + B^{N-1} \vec{c}_{\ell} \right) \cdot \left( \Psi_{q_{k}, N-1}(0) + A^{N} \vec{b}_{m} \right). \]
By Lemma \ref{lem2} Item i), for some $z \in \mathbb{Z}^d$, 
\[ B^{N-1} \vec{c}_{\ell} \cdot \Psi_{q_{k}, N-1}(0) = B^{N-1} \vec{c}_{\ell} \cdot A^{N-1} z = \vec{c}_{\ell} \cdot z \in \mathbb{Z}. \]
Note that
\[ B^{N-1} \vec{c}_{\ell} \cdot A^{N} \vec{b}_{m} = \vec{c}_{\ell} \cdot A \vec{b}_{m}. \]

Therefore, combining the above, we obtain
\begin{align}
[ M_{N} ]_{jk} &= e^{- 2 \pi i \mathcal{R}_{j,N}(0) \cdot \Psi_{k,N}(0) } \notag \\
&= e^{- 2 \pi i \mathcal{R}_{q_{j}, N-1}(0) \cdot \Psi_{q_{k}, N-1}(0) } e^{- 2 \pi i \mathcal{R}_{q_{j}, N-1}(0) \cdot A^{N} \vec{b}_{m} } e^{-2 \pi i \vec{c}_{\ell} \cdot A \vec{b}_{m} } \notag \\
&= [M_{N-1}]_{q_{j} q_{k}} e^{- 2 \pi i \mathcal{R}_{q_{j}, N-1}(0) \cdot A^{N} \vec{b}_{m} }  [M_{1}]_{\ell m}. \label{Eq:matrix2}
\end{align}
Letting $j$ vary between $\ell K^{N-1}$ and $(\ell + 1) K^{N-1}$ and $k$ vary between $m K^{N-1}$ and $(m + 1) K^{N-1}$ corresponds to $q_{j}$ and $q_{k}$ varying between $0$ and $K^{N-1}$.  Therefore, we obtain from Equations (\ref{Eq:matrix1}) and (\ref{Eq:matrix2}) the matrix equation
\[  B_{\ell m} = [M_{1}]_{\ell m} D_{N, m} M_{N-1} \]
where $[ D_{N,m} ]_{p p} = e^{- 2 \pi i \mathcal{R}_{p, N-1}(0) \cdot A^{N} \vec{b}_{m} }$ as claimed.
\end{proof}

\begin{corollary}
The matrix $M_{N}$ is invertible.  If $M_{1}$ is Hadamard, then $M_{N}$ is also Hadamard.
\end{corollary}

\begin{proof}
If $M_{1}$ is invertible, then by induction, $M_{N}$ is invertible via Proposition \ref{Hoinv}.  If $M_{1}$ is Hadamard, then again by induction, $M_{N}$ is Hadamard by \dita's construction.
\end{proof}

\begin{theorem} \label{Th:main2}
The matrix $\widetilde{M}_{N}$ representing the exponentials with frequencies given by $\mathcal{T}_{N}$ on the fractal approximation $\mathcal{S}_{N}$, when both are endowed with the reverse ordering, has the form
\begin{equation} \label{Eq:main2}
\widetilde{M}_{N} = 
\left( \begin{array}{cccc} m_{00} \widetilde{M}_{N-1} & m_{01}  \widetilde{M}_{N-1}  & \dots & m_{0(K-1)}  \widetilde{M}_{N-1}  \\
 m_{10} \widetilde{M}_{N-1} \widetilde{D}_{N,1} & m_{11}   \widetilde{M}_{N-1} \widetilde{D}_{N,1} & \dots & m_{1(K-1)}  \widetilde{M}_{N-1} \widetilde{D}_{N,1}\\
\vdots & \vdots & \vdots \ \vdots & \vdots \\
m_{(K-1)0} \widetilde{M}_{N-1}\widetilde{D}_{N,K-1} & m_{(K-1)1}   \widetilde{M}_{N-1} \widetilde{D}_{N,K-1}  & \dots & m_{(K-1)(K-1)}  \widetilde{M}_{N-1} \widetilde{D}_{N,K-1} \\
\end{array}
\right).
\end{equation}
Here, $\widetilde{D}_{N,q}$ is a diagonal matrix with $[\widetilde D_{N,\ell}]_{pp}= e^{- 2\pi i  c_\ell \cdot A (\widetilde\Psi_{p,N-1} (0))}$, and $m_{jk} = [M_1]_{jk}$. 
\end{theorem}

\begin{proof}
The proof proceeds similarly to the proof Theorem \ref{Th:main1}.  Let us first subdivide $\widetilde{M}_{N}$ into $K^{N-1} \times K^{N-1}$ blocks $\widetilde{B}_{\ell m}$, so that 
\[ \widetilde{M}_{N} = \begin{pmatrix} \widetilde{B}_{00} & \dots  & \widetilde{B}_{0 (K-1)} \\ \vdots & \ddots & \vdots \\ \widetilde{B}_{(K-1) 0} & \dots & \widetilde{B}_{(K-1) (K-1)}
\end{pmatrix}. \]
Fix $0 \leq j,k < K^{N}$ and suppose $ \ell K^{N-1} \leq j < (\ell + 1)K^{N-1}$ and $m K^{N-1} \leq k < (m+1)K^{N-1}$ with $0 \leq \ell, m < K$.  Let $q_{j} = j - \ell K^{N-1}$ and $q_{k} = k - m K^{N-1}$.  Observe that 
\begin{equation} \label{Eq:matrix3}
[ \widetilde{M}_{N} ]_{jk} = [ \widetilde{B}_{\ell m} ]_{q_{j} q_{k}}. 
\end{equation}

We calculate using Lemma \ref{lem3} items ii) and iv):
\begin{align*}
\widetilde{\mathcal{R}}_{j,N}(0) \cdot \widetilde{\Psi}_{k,N}(0) = \ & 
 ( B \widetilde{\mathcal{R}}_{q_{j},N-1}(0) + \vec{c}_{\ell} ) \cdot (A \widetilde{\Psi}_{q_{k},N-1}(0) + A \vec{b}_{m} ) \\
= \ &  \widetilde{\mathcal{R}}_{q_{j},N-1}(0) \cdot \widetilde{\Psi}_{q_{k},N-1}(0) + \vec{c}_{\ell} \cdot A \widetilde{\Psi}_{q_{k},N-1}(0) \\
& \qquad + \widetilde{\mathcal{R}}_{q_{j}, N-1}(0) \cdot \vec{b}_{m} + \vec{c}_{\ell} \cdot  A \vec{b}_{m}. 
\end{align*}
By Lemma \ref{lem2} Item iv), $\widetilde{\mathcal{R}}_{q_{j}, N-1}(0) \cdot \vec{b}_{m} \in \mathbb{Z}$.  Thus, 
\[
[ \widetilde{B}_{\ell m} ]_{q_{j} q_{k}} = [M_{N-1}]_{q_{j} q_{k}} e^{- 2 \pi i \vec{c}_{\ell} \cdot A \widetilde{\Psi}_{q_{k},N-1}(0)} [M_{1}]_{\ell m}
\]
and as in the proof of Theorem \ref{Th:main1}, we have
\[ \widetilde{B}_{\ell m} = [M_{1}]_{\ell m} \widetilde{M}_{N-1} \widetilde{D}_{N,\ell}. \]
\end{proof}

\subsection{Computational Complexity of Theorems \ref{Th:main1} and \ref{Th:main2}}

As a consequence of Proposition \ref{Hocount}, the matrix $M_{N}$ can be multiplied by a vector of dimension $K^{N}$ in at most $K \mathcal{P}_{N-1} + 3 K^{N+1} - 2 K^{N}$ operations, where $\mathcal{P}_{N-1}$ is the number of operations required by the matrix multiplication for $M_{N-1}$.  Since $M_{N-1}$ has the same block form as $M_{N}$, $\mathcal{P}_{N-1}$ can be determined by $\mathcal{P}_{N-2}$, etc.  The proof of the following proposition is a standard induction argument, which we omit.  Note that this says that the computational complexity for $M_{N}$ is comparable to that for the FFT (recognizing the difference in the number of generators for the respective IFS's).

\begin{proposition}
The number of operations to calculate the matrix multiplication $M_{N} \vec{v}$ is $\mathcal{P}_{N} = K^{N-1} \mathcal{P}_{1} + 3(N-1) K^{N+1} - 2 (N-1) K^{N}$.  Consequently, $\mathcal{P}_{N} = O(N \cdot K^{N})$.
\end{proposition}

The significance of Theorem \ref{Th:main2} concerns the inverse of $M_{N}$.  If $P$ is the permutation matrix as in Lemma \ref{L:tilde}, then $M_{N}^{-1} = P \widetilde{M}_{N}^{-1} P$.  By Proposition \ref{Hoinv}, $\widetilde{M}_{N}^{-1}$ has the form of \dita's construction, and so the computational complexity of $\widetilde{M}_{N}^{-1}$ is the same as $M_{N}$.  Thus, modulo multiplication by the permutation matrices $P$, the computational complexity of multiplication by $M_{N}^{-1}$ is the same as that for $M_{N}$.


\subsection{The Diagonal Matrices}
The matrices $M_{N}$ and $\widetilde{M}_{N}$ have the form of \dita's construction as shown in Theorems \ref{Th:main1} and \ref{Th:main2}.  The block form of \dita's construction involves diagonal matrices, which in Equations (\ref{Eq:main1}) and (\ref{Eq:main2}) are determined by the IFS's used to generate the matrices $M_{N}$ and $\widetilde{M}_{N}$.  As such, the diagonal matrices satisfy certain recurrence relations.

\begin{theorem}
The diagonal matrices which appear in the block form of $M_{N}$ (Equation (\ref{Eq:main1})) satisfy the recurrence relation $D_{N,m} = D_{N-1,m}  \otimes E_{N,m}$, where $E_{N,m}$ is the $K \times K$ diagonal matrix with $[E_{N,m}]_{uu} = e^{- 2 \pi i c_u \cdot A^{N} \vec b_m}$.  That is: 
\[ [D_{N,m}]_{pp} =  [D_{N-1,m}]_{\widehat p \widehat p} \ e^{-2 \pi i (c_{p_0} \cdot A^{N} \vec b_m)}\]
where $\widehat p = (p - p_0)/K$.

Likewise, the diagonal matrices which appear in the block form of $\widetilde{M}_{N}$ (Equation (\ref{Eq:main2})) satisfy the recurrence relation $\widetilde D_{N,\ell} = \widetilde D_{N-1,\ell}  \otimes \widetilde E_{N, \ell}$, where $\widetilde E_{N,\ell}$ is the $K \times K$ diagonal matrix with $[\widetilde E_{N,\ell}]_{uu} = e^{- 2 \pi i \vec c_\ell \cdot A^N \vec b_u}$.  That is: 
\[ [\widetilde D_{N,\ell}]_{pp}= [D_{N-1,\ell}]_{\widehat p \widehat p} \  e^{- 2\pi i  \vec c_\ell \cdot A^{N} \vec b_{p_0}}. \]
\end{theorem}

\begin{proof}
As demonstrated in Theorem \ref{Th:main1}, for $p = 0,1,\dots,K^{N-1}$, $[D_{N,m}]_{pp}= e^{- 2\pi i \mathcal{R}_{p,N-1}(0) \cdot A^N \vec b_m}$.  Note that $p_{N-1} = 0$, and $\rho_{0}(0) = 0$.  We want to cancel one power of $A$ in $A^N \vec b_m$, so we factor out a $B$ from $\mathcal{R}_{p,N-1}(0)$:
\[ \mathcal{R}_{p,N-1}(0) = \rho_{p_0} \circ \rho_{p_1} \circ \dots \circ \rho_{p_{N-2}} (0) = B \left(\rho_{p_1} \circ \dots \circ \rho_{p_{N-2}} (0)\right) + \vec c_{p_0}. \] 
Since $\widehat p =p_1 +p_2 K + \dots + p_{N-2} K^{N-3}$, $\mathcal{R}_{p,N-1}(0) =  B R_{\widehat p,N-2}(0)+ \vec c_{p_0}$.  Thus,
\begin{align*} 
[D_{N,m}]_{pp} &= e^{- 2\pi i \mathcal{R}_{p,N-1}(0) \cdot A^N \vec b_m} \\
&=e^{- 2\pi i (B \mathcal{R}_{\widehat p,N-2}(0) \cdot A(A^{N-1}\vec b_m))} e^{-2 \pi i (\vec c_{p_0} \cdot A^{N} \vec b_m)}\\
&=e^{- 2\pi i (\mathcal{R}_{\widehat p,N-2}(0) \cdot(A^{N-1}\vec b_m))} e^{-2 \pi i (\vec c_{p_0} \cdot A^{N} \vec b_m)}\\
&= [D_{N-1,m}]_{\widehat p \widehat p} \ e^{-2 \pi i (\vec c_{p_0} \cdot A^{N} \vec b_m)}. 
\end{align*}

Similarly, as demonstrated in Theorem \ref{Th:main2}, $[\widetilde D_{N,\ell}]_{pp}= e^{- 2\pi i \vec  c_\ell \cdot A (\widetilde\Psi_{p,N-1} (0))}$.  We write: 
\begin{align*}
\widetilde\Psi_{p,{N-1}}(0) &=  \psi_{p_{N-2}} \circ \psi_{p_{N-3}} \circ \dots \circ \psi_{p_1} \circ \psi_{p_0}(0) \\
&= \psi_{p_{N-2}} \circ \psi_{p_{N-3}} \circ \dots \circ \psi_{p_1}(0 + A \vec b_{p_0}) \\
&= \widetilde \Psi_{\widehat p, N-2} (0 + A \vec b_{p_0}) \\
&= \widetilde\Psi_{\widehat p, N-2} (0) + A^{N-1}\vec b_{p_0}.
\end{align*}
where in the last equality we use Lemma \ref{lem1} item ii).  Therefore: 
\begin{align*}
[\widetilde D_{N,\ell}]_{pp} &= e^{- 2\pi i  c_\ell \cdot A (\widetilde\Psi_{p,N-1} (0))}\\
&= e^{- 2\pi i  \vec c_\ell \cdot A \left( \widetilde\Psi_{\widehat p, N-2} (0) + A^{N-1} \vec b_{p_0}\right) } \\
&= e^{- 2\pi i  \vec c_\ell \cdot A  \left(\widetilde\Psi_{\widehat p, N-2} (0) + A^{N}\vec b_{p_0}\right) }\\
&= e^{- 2\pi i  \vec c_\ell \cdot A ( \widetilde\Psi_{\widehat p, N-2} (0))} e^{- 2\pi i  \vec c_\ell \cdot A^{N} \vec b_{p_0}}\\
&= [\widetilde D_{N-1,\ell}]_{\widehat p \widehat p} \ e^{- 2\pi i  \vec c_\ell \cdot A^{N} \vec b_{p_0} }.
\end{align*}

\end{proof}

\bibliographystyle{amsplain}
\bibliography{ffft}
\nocite{*}

\end{document}